\documentclass[reqno, 11pt]{amsart}
\usepackage{pdfsync}
\usepackage{url}
\usepackage[colorlinks, linkcolor=blue, citecolor=blue, urlcolor=blue]{hyperref}
\usepackage{times}

\newtheorem{theorem}{Theorem}[section]
\newtheorem{lemma}[theorem]{Lemma}

\newtheorem{corollary}[theorem]{Corollary}

\theoremstyle{definition}

\newtheorem{remark}[theorem]{Remark}

\numberwithin{equation}{section}

\usepackage{bbm}
\usepackage{url}
\usepackage{euscript}
\usepackage{pb-diagram}
\usepackage{lamsarrow}
\usepackage{pb-lams}
\usepackage{amsmath}
\usepackage{amsthm}

\usepackage{amsxtra}
\usepackage{amssymb}
\usepackage{pifont}
\usepackage{amsbsy}

\usepackage{graphicx}
\usepackage{epstopdf}

\newskip\aline \newskip\halfaline
\def\skipaline{\vskip\aline}

\def\qedbox{$\rlap{$\sqcap$}\sqcup$}
\def\qed{\nobreak\hfill\penalty250 \hbox{}\nobreak\hfill\qedbox\skipaline}

\def\proofend{\eqno{\mbox{\qedbox}}}

\newcommand\bR{{\mathbb R}}

\newcommand\bZ{{\mathbb Z}}

\newcommand{\ii}{\boldsymbol{i}}

\newcommand{\bp}{{\boldsymbol{p}}}
\newcommand{\bq}{{\boldsymbol{q}}}

\newcommand{\bv}{{\boldsymbol{v}}}

\newcommand{\bsB}{\boldsymbol{B}}

\newcommand{\bsE}{\boldsymbol{E}}

\newcommand{\bsU}{{\boldsymbol{U}}}

\newcommand{\bsi}{\boldsymbol{\sigma}}

\newcommand{\si}{{\sigma}}

\newcommand{\ve}{{\varepsilon}}

\newcommand{\vfi}{{\varphi}}

\newcommand{\eA}{\EuScript{A}}

\newcommand{\eD}{\EuScript{D}}
\newcommand{\eE}{\EuScript{E}}

\newcommand{\eH}{\EuScript H}

\newcommand{\eM}{\EuScript{M}}
\newcommand{\eN}{\EuScript{N}}
\newcommand{\eO}{\EuScript{O}}

\newcommand{\eU}{\EuScript{U}}

\newcommand{\ra}{\rightarrow}

\def\inpr{\mathbin{\hbox to 6pt{\vrule height0.4pt width5pt depth0pt \kern-.4pt \vrule height6pt width0.4pt depth0pt\hss}}}

\newcommand{\pa}{\partial}

\begin{document}

\title[The blowup of the spectral function]{The blowup along the diagonal of the spectral function  of the Laplacian} 

\date{Started on February 2, 2011. Completed on March 2, 2011. Latest revision \today}

\author{Liviu I. Nicolaescu}
\thanks{This work was partially supported by the NSF grant, DMS-1005745.}

\address{Department of Mathematics, University of Notre Dame, Notre Dame, IN 46556-4618.}

\email{nicolaescu.1@nd.edu}
\urladdr{\url{http://www.nd.edu/~lnicolae/}}

\subjclass[2000]{58J50, 35J15, 33C45, 32C05}
\keywords{Riemann manifolds, Laplacian, eigenfunctions, spectral function, real analytic manifolds,  harmonic polynomials}

\begin{abstract}  We formulate  a precise conjecture about  the universal behavior  near the diagonal of the spectral function of the Laplacian    of a smooth compact Riemann manifold. We prove  this conjecture   when  the manifold and the metric are real analytic, and we also present an alternate  proof      when  the manifold is the round sphere.

\end{abstract}

\maketitle

\tableofcontents

\section{Introduction}
\setcounter{equation}{0}

Suppose that $(M,g)$ is a compact, connected $m$-dimensional Riemannian manifold,  and $(\Psi_n)_{n\geq 0}$ is a (complete) orthonormal basis of $L^2(M,g)$ consisting of eigenfunctions of $\Delta_g$
\[
\Delta_g \Psi_n=\lambda_n\Psi_n,\;\; 0=\lambda_0<\lambda_1\leq \lambda_2\leq \cdots\leq \lambda _n\leq \cdots.
\]
For every $L>0$ we define the spectral function  $\eE_L:M\times M\ra \bR$ by
\[
\eE_L(\bp,\bq)=\sum_{\lambda_n\leq L} \Psi_n(\bp)\Psi_n(\bq).
\]
Equivalently, $\eE_L$ is the Schwartz kernel of the orthogonal projection onto
\[
H_L:=\bigoplus_{\lambda\leq L} \ker(\lambda-\Delta).
\]
This shows that as $L\ra \infty$   the spectral function $\eE_L$ converges in the sense of distributions to the Dirac-type distribution supported by the diagonal
\[
\eD_M=\bigl\{ (\bp,\bq)\in M\times M;\;\;\bp=\bq\,\bigr\}.
\]
The goal of this paper is to  describe  a universal law governing  the behavior  of $\eE_L$  as $L\ra \infty$ in an infinitesimal  neighborhood of the diagonal. Here are the specifics.

We denote by $\eN$ the normal bundle    of the diagonal embedding.  For any $\bp\in M$ we denote by $\eN_\bp$ the fiber over $(\bp,\bp)\in \eD_M$.  Also  we  let $r:\eN\ra \bR$ denote the radial distance function along the fibers of $\eN$, and we set
\[
\eN^R:=  r^{-1}\bigl([0,R)\,\bigr),\;\;\eN^R_\bp:=\eN^R\cap\eN_\bp,\;\;\forall R>0,\;\;\bp\in M.
\]
In other words,  $\eN^R\subset \eN$ is  the associated bundle   of normal disks of radius $R$.  If  $\hbar$ is sufficiently small, then  the exponential map  induces a diffeomorphism    from $\eN^\hbar$  onto an open neighborhood $\eU^\hbar$ of the diagonal.   Fix once and for all such a $\hbar$.    We denote by $\eE_L^\hbar$ the pullback of $\eE_L|_{\eU^\hbar}$ to $\eN^\hbar$. 
 
For every  positive real number $\lambda$ we denote by $\eM_\lambda: \eN\ra \eN$ the   rescaling  map described on $\eN_\bp$ by
\[
\eN_\bp\ni\bv \mapsto \frac{1}{\lambda} \bv \in \eN_\bp.
\]
We  define 
\[
\bar{\eE}_L:\eN^{L^{1/2}\hbar}\ra \bR,\;\;\bar{\eE}_L=  L^{-\frac{m}{2}} \eM_{L^{\frac{1}{2}}}^* \eE_L^\hbar.
\]
For any $\bp\in M$ we denote by $\bar{\eE}_{L,\bp}$ the restriction of $\bar{\eE}_L$ to the fiber $\eN_\bp$. \medskip

\noindent\textbf{\textsl{The Universality Conjecture. }}    There exists $\rho>0$ such that  for  any $\bp\in M$ the functions  
\[
\bar{\eE}_{L,\bp}:\eN_\bp^{L^{\frac{1}{2}}\hbar}\ra \bR
\]
  converge  as  $L\ra \infty$ in the topology of $C^\infty(\eN_\bp^\rho)$  to  the  smooth function 
  \[
  E_\infty: \eN_\bp \ra \bR,\;\;E_\infty(u)=\frac{1}{(2\pi |u|)^{ \frac{m}{2} } } J_{\frac{m}{2}}(|u|),
  \]
where $J_\nu$ denotes the Bessel function  of the first kind and order $\nu$.  
\qed

\begin{remark}\noindent (a) The limit function  $E_\infty(u)$ has a more suggestive description, namely
\begin{equation}
E_\infty(u)=\frac{1}{(2\pi|u|)^{\frac{m}{2}}}J_{\frac{m}{2}}(|u|)=\frac{1}{(2\pi)^m}\int_{\bsB_1^m} e^{\ii (\xi,u)} |d\xi|,\;\;\bsB_r^m:=\bigl\{\xi\in\bR^m;\;\;|\xi|^2\leq r\,\bigr\}.
\label{eq: einfty}
\end{equation}
We denote by $\bsE_L$ the spectral function of the Laplacian on $\bR^m$ corresponding to (generalized) eigenvalues $\leq L$.   We then have (see \cite[Eq. (2.1)]{Lai} or \cite[Eq.(1.3)]{Pee})
\[
\bsE_L(x,y)=\frac{1}{(2\pi)^m}\int_{\bsB^m_L} e^{\ii(\xi, x-y)} |d\xi|.
\]
This shows that
\[
E_\infty(u)=\bsE_1(u,0).
\]

\smallskip

(b) We   ought to  elaborate on the  crux of the Universality conjecture.     First of all, let us point out that known local Weyl trace formul{\ae} (\cite{Bin}, \cite[\S 17.5]{Ho3}, \cite[Thm. 1.8.5]{SV}, \cite[\S 7,8]{Zel}) imply that the family $(\bar{\eE}_L\,)_{L\geq 1}$ is precompact  in the  $C^\infty$-topology and   as detailed in see   Section \ref{s: real-an}, any limit point $\bar{\eE}_\infty$  is asymptotically equivalent to $E_\infty$ at $u=0$, i.e., for any $N>0$,
\[
\bar{\eE}_{\infty,\bp}(u)=E_\infty(u) +O(|u|^{-N})\;\;\mbox{as}\;\;u\searrow 0.
\]
The    Universality Conjecture makes the stronger claim  that the family $(\bar{\eE}_L\,)_{L\geq 1}$  has a \emph{unique}  limit point in $C^\infty$-topology and moreover, that limit point is \emph{not just asymptotically equivalent,  but equal} to  the function $E_\infty$.  

(c) Recently, Lapointe, Polterovich and Safarov  \cite{LPS} have described a global  type of  relationship  between  $\eE_L$ and $\bsE_L$  as $L\ra \infty$.

\qed
\label{rem: hor}
\end{remark}

The main result   of this paper  states that  the Universality Conjecture is true in the case when   $M$ and $g$ are real analytic.  We achieve this in  Section \ref{s: real-an} by relying on some \emph{sharp} a priori estimates that we prove in  Appendix \ref{s: ell-est}, Theorem \ref{th: fund}.  A weaker version  of these estimates  were used   by Donnelly and Fefferman \cite[\S 7]{DF} in their investigation of nodal sets of eigenfunctions on real analytic Riemannian manifolds. We believe that Theorem \ref{th: fund}  will find many other uses. In Section  \ref{s: sph} we  give an alternate proof  of the conjecture in   the special case when $(M,g)$ is the round sphere.

\smallskip

\noindent {\bf Acknowledgments.} I  want to thank  Steve Zelditch for a most illuminating discussion  on    spectral geometry.

\section{The Universality Conjecture in the real analytic case}
\setcounter{equation}{0}
\label{s: real-an}

\begin{theorem}  The Universality Conjecture is true  when $M$ and $g$ are real analytic. 
\label{th: conj}
\end{theorem}

\begin{proof} Since both $M$ and $g$ are real analytic we deduce that the  spectral function $\eE_L$ is real analytic; see \cite{LM, Mo}.    The  Cauchy-Kowaleskaya theorem,  \cite{KP}, implies  that the exponential map  is also real analytic so that $\eE_L^\hbar$ is real analytic.

 Fix  a point $\bp_0$ in $M$ and normal coordinates $x=(x^1,\dotsc, x^m)$  at $\bp_0$ defined on an open neighborhood $\eO$ of $\bp_0$.    With these choices we can  regard the restriction of  $\eE_L$ to  $\eO\times \eO$  as a real analytic function $\eE_L(x,y)$ defined on a neighborhood of $(0,0)$ in $\bR^m\times \bR^m$.        Via the exponential map 
\[
\exp_{(\bp_0,\bp_0)} : T_{(\bp_0,\bp_0)} M\times M\ra M\times M
\]
 we can  identify the    space $\bigl\{ (x,y)\in\bR^m\times \bR^m;\;\;x+y=0\,\bigr\}$ with the   fiber $\eN_{\bp_0}$.

The results  of  \cite{Bin} show that as $L\ra \infty$ we have
\begin{equation}
\frac{\pa ^{\alpha+\beta}}{\pa x^\alpha \pa y^\beta}{\eE}_L(x,y)_{x=y=0}= L^{\frac{m+|\alpha|+|\beta|}{2}}\Bigl(C_{\alpha,\beta}+ O\bigl(\,L^{-\frac{1}{2}}\,\bigr)\,\Bigr),
\label{eq: 1}
\end{equation}
where $C_{\alpha,\beta}=0$ if $\alpha-\beta\not\in 2\bZ^m$,    while if $\alpha-\beta\in 2\bZ^m$ we have
\[
C_{\alpha,\beta}=(-1)^{\frac{|\alpha|-|\beta|}{2}}\frac{1}{(4\pi)^{\frac{m}{2}} 2^{\frac{|\alpha|+|\beta|}{2}}\Gamma\bigl(\, 1+\frac{m}{2}+\frac{|\alpha|+|\beta|}{2}\,\bigr)}\prod_{i=1}^m( \alpha_i+\beta_i-1)!!.
\]
In fact we can say a bit more. More precisely, according to \cite[Thm. 17.5.3]{Ho3},  for any $\alpha,\beta$  there exists a  constant $K_{\alpha,\beta}>0$  that depends on the geometry of $(M,g)$ but it is independent of $\bp_0$ such that

\begin{equation}
\left|\frac{\pa ^{\alpha+\beta}}{\pa x^\alpha \pa y^\beta}{\eE}_L(x,y)\,\right|\leq K_{\alpha,\beta} L^{\frac{m+|\alpha|+|\beta|}{2}},\;\;\forall x,y.
\label{eq: 1.1}
\end{equation}
A priori,  the constants  $K_{\alpha,\beta}$ can grow really fast   as $|\alpha|+|\beta|\ra \infty$. Our next result   provides  a key upper bound on this growth. To keep the flow of arguments uninterrupted we deferred its rather  sneaky proof to Appendix \ref{s: ell-est}.

\begin{lemma} There exist  constant $C, T>0$  such   that for any   $L>1$  and any multi-indices $\alpha$, $\beta$ we have
\begin{equation}
\sup_{(\bp,\bq)\in M\times M} |\pa^{\alpha+\beta}_{x,y}  \eE_L(\bp,\bq)|\leq  C T^{|\alpha|+|\beta|} \bigl(\,|\alpha|+|\beta|\,\bigr)! L^{\frac{m+|\alpha|+|\beta|}{2}},
\label{eq: control}
\end{equation}
where $(x,y)$ denote    normal coordinates at $(\bp,\bq)$.  In other words,  in (\ref{eq: 1.1}) we can choose constants $K_{\alpha,\beta}$ satisfying
\begin{equation}
K_{\alpha,\beta}\leq K T^{|\alpha|+|\beta|} (|\alpha|+|\beta|)!,
\label{eq: control1}
\end{equation}
where $K$ and $T$ are independent of  $\alpha, \beta$ and $L$.\qed
\label{lemma: control}
\end{lemma}

Introduce   new coordinates 
\[
\eta^i:=\frac{1}{\sqrt{2}}(x^i-y^i), \;\;\tau^j:=\frac{1}{\sqrt{2}}(x^j+y^j),\;\; i,j=1,\dotsc, m.
\]
The  fiber $\eN_{\bp_0}$ is described by the  equations $\tau^j=0$, $j=1,\dotsc, m$.  Note that
\[
x^j =\frac{1}{\sqrt{2}}(\eta^j+\tau^j),\;\; y^j=\frac{1}{\sqrt{2}}(\tau^j-\eta^j).
\]
Along $\eN_{\bp_0}$  we can use the functions $(\eta^j)$ as orthonormal coordinates. We have 
\[
r^2=\sum_i(\eta^i)^2,
\]
and
\[
\eE_{L,\bp_0}= \eE_L( 2^{-\frac{1}{2}} \eta, -2^{-\frac{1}{2}}\eta),\;\;\bar{\eE}_{L,\bp_0}(\eta)= L^{-\frac{m}{2}}\eE_L\left( \frac{1}{(2L)^{\frac{1}{2}}} \eta, -\frac{1}{(2L)^{\frac{1}{2}}}\eta\right)
\]
We set
\[
C_{\alpha}(L):=\frac{\pa^\alpha}{\pa_\eta^\alpha}\bar{\eE}_{L,\bp_0}(\eta)|_{\eta=0}.
\]
From (\ref{eq: control1}) we deduce that there exist constants $K, T>0$  such that for any  multi-index $\alpha$ and any  $L>1$ we have 
\begin{equation}
\left|C_{\alpha}(L)\right|\leq K|\alpha|! T^{|\alpha|}.
\label{eq: 1.2}
\end{equation}
Using (\ref{eq: 1}) we deduce that
\[
\pa^\ell_{\eta^i} \bar{\eE}_L(x,y)_{x=y=0}= C_\ell+O\bigl(\, L^{-\frac{1}{2}}\,\bigr),
\]
where $C_\ell=0$ if $\ell$ is odd, while if $\ell$ is even we have
\[
C_\ell=2^{-\frac{\ell}{2}}\sum_{k=0}^\ell (-1)^k \binom{\ell}{k} \pa^{\ell-k}_{x^i} \pa^k_{y^i} \bar{\eE}_L(x,y)_{x=y=0}=(-1)^{\frac{\ell}{2}} \frac{(\ell-1)!!}{(4\pi)^{\frac{m}{2}} \Gamma( 1+\frac{m}{2}+\frac{\ell}{2})}.
\]
More generally,
\begin{equation}
C_{\alpha}(L):= O\bigl(L^{-\frac{1}{2}}\,\bigr)\;\;\mbox{if}\;\;\alpha\not\in 2\bZ,
\label{eq: 1.3}
\end{equation}
and 
\begin{equation}
C_{2\alpha}(L)=(-1)^{|\alpha|} \frac{1}{(4\pi)^{\frac{m}{2}} \Gamma( 1+\frac{m}{2}+|\alpha|)}\prod_{j=1}^m (2\alpha_j-1)!!+ O\bigl(L^{-\frac{1}{2}}\,\bigr).
\label{eq: 1.4}
\end{equation}
The   estimates  (\ref{eq: 1.2}),  (\ref{eq: 1.3}) and (\ref{eq: 1.4}) show that  there exists $\rho>0$  such  the family $\bigl\{ \bar{\eE}_L\,\bigr\}_{L\geq 1}$ of real analytic functions   converges   in the topology of $C^\infty(\{|\eta|<\rho\})$ to a function $\eE_\infty$ that is real analytic on the ball $\{|\eta|<\rho\}$. Moreover,
\begin{equation}
\pa^{2\alpha}_\eta\eE_\infty(0)=(-1)^{|\alpha|} \frac{1}{(4\pi)^{\frac{m}{2}} \Gamma( 1+\frac{m}{2}+|\alpha|)}\prod_{j=1}^m (2\alpha_j-1)!!.
\label{eq: 2}
\end{equation}
We deduce that for $|\eta|<\rho$ we have
\[
\eE_\infty(\eta)=\sum_\alpha \frac{1}{\alpha!}\pa^\alpha_\eta \eE_\infty(0)\eta^\alpha=\sum_{n\geq 0}  \frac{(-1)^n}{(4\pi)^{\frac{m}{2}} \Gamma( 1+\frac{m}{2}+n)}\sum_{|\alpha|=n}\frac{\prod_j(2\alpha_j-1)!}{\prod_j  (2\alpha_j)!} \eta^{2\alpha}
\]
\[
=\sum_{n\geq 0}  \frac{(-1)^n}{(4\pi)^{\frac{m}{2}} \Gamma( 1+\frac{m}{2}+n)}\sum_{|\alpha|=n}\frac{1}{2^{|\alpha|} \alpha!} \eta^{2\alpha}=\sum_{n\geq 0}  \frac{(-1)^n}{(4\pi)^{\frac{m}{2} } 2^n n!\Gamma( 1+\frac{m}{2}+n)}\sum_{|\alpha|=n}\frac{n!}{\prod_j \alpha_j!} \eta^{2\alpha}
\]
\[
=\sum_{n\geq 0}  \frac{(-1)^n}{(4\pi)^{\frac{m}{2} } 2^n n!\Gamma( 1+\frac{m}{2}+n)}\Bigl(\sum_i(\eta^i)^2\,\Bigr)^n=\sum_{n\geq 0}  \frac{(-1)^n}{(4\pi)^{\frac{m}{2} } n!\Gamma( 1+\frac{m}{2}+n)}\Bigl(\frac{r^2}{2}\Bigr)^n
\]
\[
=\frac{1}{(2\pi r)^{\frac{m}{2}} }\times \underbrace{ \Bigl(\frac{r}{2}\Bigr)^{\frac{m}{2}}\sum_{n\geq 0}  \frac{(-1)^n}{n!\Gamma( 1+\frac{m}{2}+n)}\Bigl(\frac{r^2}{2}\Bigr)}_{=: J_{\frac{m}{2}}(r)}.
\]
\end{proof}

\section{The spectral function of a round sphere}
\setcounter{equation}{0}
\label{s:   sph}

To appreciate the complexity involved in the Universality Conjecture we believe  that it is instructive to give an alternate proof of the conjecture in the special case when $M$ is   the round sphere $S^{d-1}\subset \bR^d$.  

 The spectrum of the Laplacian on $S^{d-1}$ is  (see \cite{Mu})
\[
\lambda_n= n(n+d-2),\;\;n=0,1,2,\dotsc
\]
and
\[
\dim \ker(\lambda_n-\Delta)= \mu_n=\frac{2n+d-2}{n+d-2}\binom{n+d-2}{d-2}.
\]
We set
\[
\eH_m:=\ker  (\lambda_m-\Delta),\;\;\bsU_n=\bigoplus_{k=0}^{n} \eH_m.
\]
As is well known, the   space  $\eH_k$ coincides with the space of restrictions to   $S^{d-1}$ of harmonic polynomials  of  degree $k$ in $d$ variables.   Fix  an   is an orthonormal basis $(\Psi_{k, m})_{1\leq \alpha\leq \mu_m}$ of  $\eH_m$   and set
\[
\eE_n(\bp,\bq):=\eE_{\lambda_n}(\bp,\bq)=\sum_{m=0}^n \sum_{k=1}^{\mu_m}\Psi_{k,m}(\bp)\Psi_{k,m}(\bq).
\]
The addition formula  for spherical harmonics \cite{Mu} implies that
\begin{equation}
\sum_{k=1}^{\mu_m} \Psi_{k,m}(\bp)\Psi_{k,m}(\bq)=\frac{\mu_m}{\bsi_{d-1}} P_{m,d}(\bp\bullet\bq),\;\;\forall   \bp,\bq\in S^{d-1},
\label{eq: add}
\end{equation}
where $\bp\bullet \bq$ is the  canonical inner product of $\bp,\bq\in \bR^d$,  $\bsi_{d-1}$ is the area of $S^{d-1}$,
\begin{equation}
\bsi_{d-1}=\frac{2\pi^{\frac{d}{2}}}{\Gamma\bigl(\frac{d}{2}\bigr)},
\label{eq: bsi}
\end{equation}
and $P_{n,d}$ is the Legendre polynomial  of degree $n$ and order $d$,
\[
P_{n,d}(t): =(-1)^n\frac{1}{2^n [n+\frac{d-3}{2}]_n }(1-t^2)^{-\frac{d-3}{2}}\frac{d^n}{dt^n}(1-t^2)^{n+\frac{d-3}{2}},
\]
\[
[x]_k:= x(x-1)\cdots \bigl(\,x-(k-1)\,\bigr)=\frac{\Gamma(x+1)}{\Gamma(x+1-k)}.
\]
The collection   $(P_{n,d})_{n\geq 0}$ is an orthogonal family of  polynomials with respect to the  measure $ w(t)dt=(1-t^2)^{\frac{d-3}{2}}dt$ on  $[-1,1]$.  More precisely we have, \cite[\S 2]{Mu},
\[
\int_{-1}^1  P_{n,d}(t) P_{m,d}(t) w(t) dt =  h_n \delta_{n,m},\;\;h_n=\frac{\bsi_{d-1}}{\bsi_{d-2}\mu_n}.
\]
Observe that we can rephrase (\ref{eq: add}) as
\begin{equation}
\sum_{k=1}^{\mu_m}\Psi_{k,m}(\bp)\Psi_{k,m}(\bq)=\frac{P_{m,d}(\bp\bullet\bq)}{\bsi_{d-2}h_m}.
\label{eq: add1}
\end{equation}
We denote by $k_n$ the leading coefficient of $P_{n,d}$. Its precise value is known, \cite[Eq. (7.6)]{Mu}, but all we need in the sequel is the equality \cite[Eq. (7.7)]{Mu}
\[
\frac{k_{n-1}}{k_{n}}=\frac{n+d-3}{2n+d-4}.
\]
We recall the Christoffel-Darboux formula, \cite[\S 5.2]{AAR},
\begin{equation}
\sum_{m=0}^{n} \frac{P_{m,d}(s)P_{m,d}(t)}{h_m}= \frac{k_{n}}{k_{n+1}}\frac{ P_{n+1,d}(t)P_{n,d}(s)-P_{n+1,d}(s) P_{n,d}(t)}{(t-s)h_{n}}
\label{eq: CD}
\end{equation}
In (\ref{eq: CD}) we let $s=1$. Using the equality $P_{m,d}(1)=1$, $\forall m$, we deduce
\begin{equation}
\sum_{m=0}^{n} \frac{P_{m,d}(t)}{h_m}=\frac{k_{n}}{k_{n+1}}\frac{P_{n+1,d}(t)-P_{n,d}(t)}{h_{n}(t-1)}.
\label{eq: CD1}
\end{equation}
Suming (\ref{eq: add1})  for $m=0,\dotsc, n$ we deduce from (\ref{eq: CD1}) that
\[
\eE_n(\bp,\bq)=\frac{k_{n}}{k_{n+1}}\frac{ P_{n+1,d}(t)-P_{n,d}(t) }{ \bsi_{d-2}h_{n}(t-1) }= \frac{ \mu_{n} }{ \bsi_{d-1} }\frac{ k_{n} }{k_{n+1}}\frac{P_{n+1,d}(t)-P_{n,d}(t)}{(t-1)}.
\]
Hence
\begin{equation}
\eE_n(\bp,\bq)=\eE_n(\vfi):=\frac{ \mu_{n} }{ \bsi_{d-1} }\frac{ k_{n} }{k_{n+1}}\frac{P_{n+1,d}(\cos\vfi)-P_{n,d}(\cos\vfi)}{(\cos\vfi-1)},\;\;\cos\vfi=\bp\bullet\bq.
\label{eq: spec}
\end{equation}
Observe that   $\vfi$ is the geodesic distance between $\bp$ and $\bq$. Now set $r_n:=\sqrt{\lambda_n}$ and define
\[
\bar{\eE}_n(\vfi):=r_n^{-(d-1)}\eE_n\Bigl(\frac{\vfi}{r_n}\Bigr) =\frac{ \mu_{n} }{ \bsi_{d-1}r_n^{d-2} }\frac{ k_{n} }{k_{n+1}}\frac{P_{n+1,d}(\cos\frac{\vfi}{r_n})-P_{n,d}(\cos\frac{\vfi}{r_n})}{r_n(\cos\frac{\vfi}{r_n}-1)}.
\]
Taking into account Remark \ref{rem: hor}, we see that the Universality Conjecture  will follow  from the equality
\begin{equation}
\bar{\eE}_n(\vfi)= \frac{1}{(2\pi\vfi)^{\frac{d-1}{2}}} J_{\frac{d-1}{2}}(\vfi) +O(n^{-1}),\;\;\mbox{uniformly for}\;\;|\vfi|\leq \frac{\pi}{4}.
\label{eq: univ}
\end{equation}
Observe first that   
\begin{equation}
P_{n,d}(t)=\frac{1}{\binom{n+\alpha}{n}}P_n^{(\alpha,\alpha)}(t),\;\;\alpha=\frac{d-3}{2},\;\;\binom{n+\alpha}{n}=\frac{\Gamma(n+\alpha+1)}{\Gamma(\alpha+1)\Gamma(n+1)},
\label{eq: jac}
\end{equation}
where $P_n^{(\alpha,\beta)}$ are the Jacobi  polynomials defined in \cite[\S 4.1]{Sze}.    To prove (\ref{eq: univ})   we  will use the Hilb type  asymptotic estimate for Jacobi polynomials., \cite[Eq. (8.2.17)]{Sze} or \cite[(29)]{Rau}. Here are the details.  Set 
\[
\theta:=\frac{\vfi}{r_n}
\]
Observe that
\begin{equation}
\frac{1}{\binom{n+\alpha}{n}}=O(n^{-\alpha}).
\label{eq: jac1}
\end{equation}
The Hilb type estimate  \cite[Eq. (8.2.17)]{Sze}  coupled with (\ref{eq: jac})  and (\ref{eq: jac1}) yields
\[
\begin{split}
P_{n,d} (\cos \theta) & = \Gamma(\alpha+1)\left(\frac{\theta}{\sin\theta}\right)^{\frac{1}{2}}\frac{2^\alpha}{a_n^\alpha \sin^\alpha\theta} J_\alpha(a_n\theta) + \frac{\theta^{\alpha+2}}{\sin^\alpha \theta} O(1)\\
&=\Gamma(\alpha+1)\left(\frac{\theta}{\sin\theta}\right)^{\frac{1}{2}}\frac{(2\theta)^\alpha}{ \sin^\alpha\theta} \frac{1}{(a_n\theta)^{\alpha}}J_\alpha(a_n\theta) + \frac{\theta^{\alpha+2}}{\sin^\alpha \theta} O(1),
\end{split}
\]
where  $\alpha$  is as in (\ref{eq: jac}) and  $a_n:= n+\frac{d-2}{2}$. We set
\[
 F_\alpha(x):=\frac{1}{x^\alpha} J_\alpha(x)
 \]
 and we deduce that
 \[
 P_{n,d} (\cos \theta) =\Gamma(\alpha+1)\left(\frac{\theta}{\sin\theta}\right)^{\frac{1}{2}}\frac{(2\theta)^\alpha}{ \sin^\alpha\theta} F_\alpha(a_n\theta)+ \frac{\theta^{\alpha+2}}{\sin^\alpha \theta} O(1).
 \]
 Hence
 \[
 \frac{P_{n+1,d}(\cos\theta)-P_{n,d}(\cos\theta)}{r_n(\cos\theta-1)}=-\Gamma(\alpha+1)\frac{P_{n+1,d}(\cos\theta)-P_{n,d}(\cos\theta)}{2r_n\sin^2(\frac{\theta}{2})}
 \]
 \[
 =-\Gamma(\alpha+1)\left(\frac{\theta}{\sin\theta}\right)^{\frac{1}{2}}\frac{(2\theta)^\alpha}{ \sin^\alpha\theta}\Bigl( \frac{F_{\alpha}(a_{n+1}\theta) -F_\alpha(a_n\theta)}{2r_n\sin^2(\frac{\theta}{2})}\,\Bigr)+\frac{\theta^{\alpha+2}}{r_n\sin^{\alpha+2} \theta} O(1).
 \]
Observe that $a_{n+1}\theta= a_n\theta+\theta$ and we deduce 
\[
F_\alpha(a_{n+1}\theta)-F_{\alpha}(a_n\theta)= F'_\alpha(a_n\theta)\theta + O(\theta)^2.
\]
 The  classical  identity involving   Bessel functions, \cite[Eq. (4.6.2)]{AAR},
\[
\frac{d}{dx} \Bigl(x^{-\nu} J_\nu(x)\Bigr)=- x^{-\nu} J_{\nu+1}(x),
\]
implies that  $F'_\alpha(z)=-x^{-\alpha}J_{\alpha+1}(x)$. Now observe that
\[
2r_n\sin^2\Bigl(\frac{\theta}{2}\Bigr)=\frac{r_n\theta^2}{2}\bigl(\, 1- O(\theta^2)\,\bigr).
\]
Since  
\[
O(\theta)=O(r_n^{-1})=O(n^{-1})
\]
we deduce
\[
\frac{P_{n+1,d}(\cos\theta)-P_{n,d}(\cos\theta)}{r_n(\cos\theta-1)}= -\left(\frac{\theta}{\sin\theta}\right)^{\frac{1}{2}}\frac{2^{\alpha+1}\theta^\alpha}{ \sin^\alpha\theta} \frac{F'_{\alpha}(a_n\theta)}{r_n\theta} + O(n^{-1})
\]
\[
=-\frac{2^{\alpha+1}}{\vfi}F_\alpha'(\vfi)+O(n^{-1})= \left(\frac{r}{2}\right)^{-(\alpha+1)} J_{\alpha+1}(r) + O(n^{-1}).
\]
Hence
\[
\bar{\eE}_n(\vfi)= \Gamma(\alpha+1)\frac{ \mu_{n} }{ \bsi_{d-1}r_n^{d-2} }\frac{ k_{n} }{k_{n+1}} \left(\frac{r}{2}\right)^{-(\alpha+1)} J_{\alpha+1}(r) + O(n^{-1}).
\]
The equality  (\ref{eq: univ}) now follows from the estimates
\[
\frac{\mu_n}{r_n^{d-2}}=\frac{2}{(d-2)!} +O(n^{-1}),\;\;\frac{k_n}{k_{n+1}}=\frac{1}{2}+O(n^{-1}),
\]
 the  equality (\ref{eq: bsi}), and the doubling formula
 \[
\sqrt{\pi} \Gamma(2x)=2^{2x-1}\Gamma(x)\Gamma\bigl(x+\frac{1}{2}\,\bigr).\proofend
\]
\begin{figure}[h]
 \includegraphics[height=5cm, width=6cm]{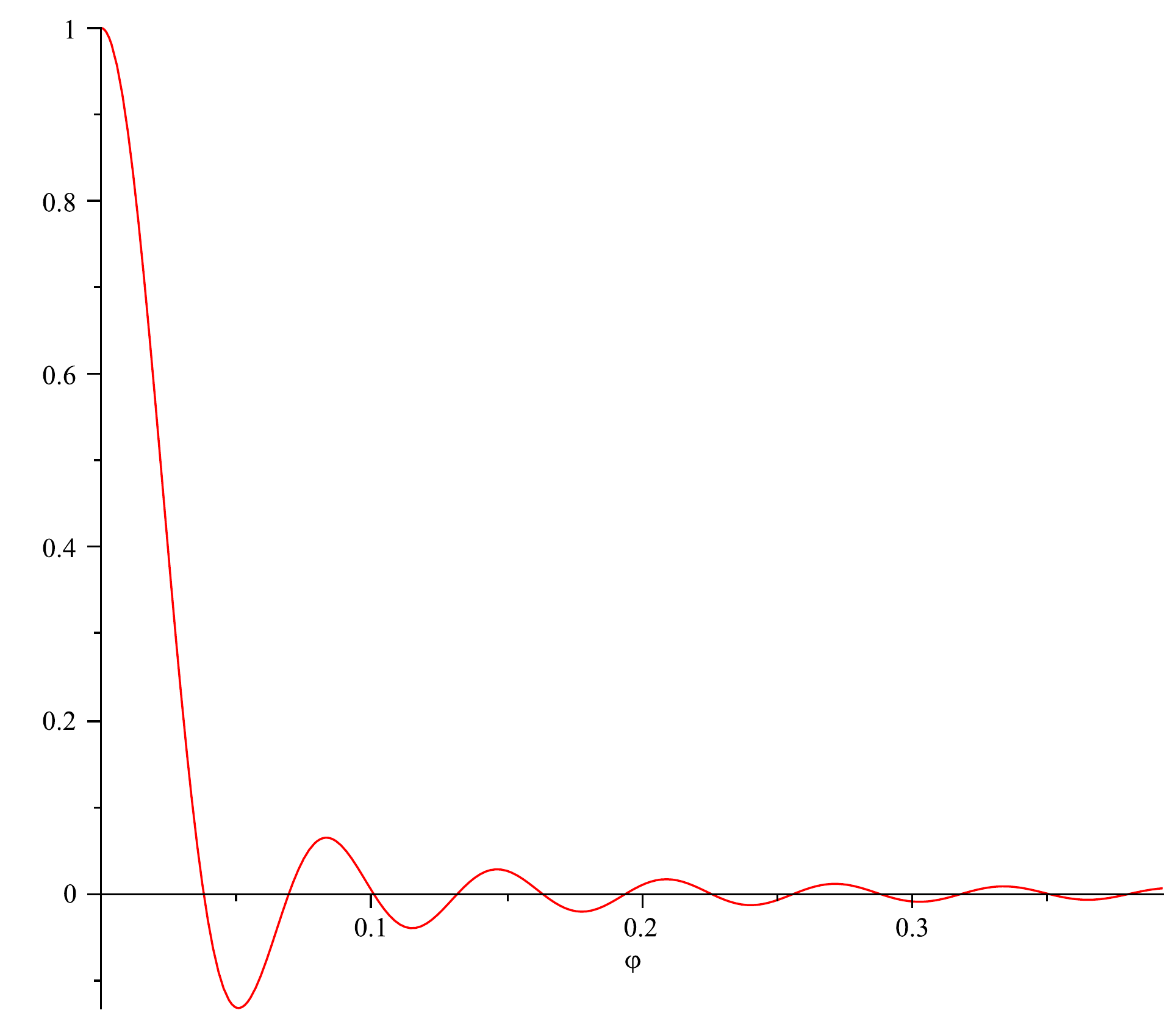}
 \caption{A depiction of $\frac{1}{\eE_{100}(0)}\eE_{100}(\vfi)$, $0\leq \vfi\leq \frac{\pi}{8}$. }
 \label{fig: spec1}
 \end{figure}

\begin{remark}  In the special case $M=S^2$ we can visualize the blowup behavior of the spectral function. In this case the eigenvalues  of the Laplacian are $\lambda_n=n(n+1)$ and 
\[
\bar{\eE}_n(\vfi)=\frac{1}{\lambda_n}\eE_n\left(\frac{\vfi}{\sqrt{\lambda_n}}\right).
\]
The  function $\eE_n(\vfi)$ has a peak at $\vfi=0$, 
\[
\eE_n(0)=\frac{(n+1)^2}{4\pi}\sim \frac{1}{4\pi}\lambda_n\;\;\mbox{as $n\ra \infty$}.
\]
The   oscillations   in the graph  depicted in Figure \ref{fig: spec1} are pushed  at $\infty$  by the rescaling in $\vfi$, and the  resulting behavior  becomes rather tame, Figure \ref{fig: spec2}.
 
 \begin{figure}[h]
 \includegraphics[height=5cm, width=6cm]{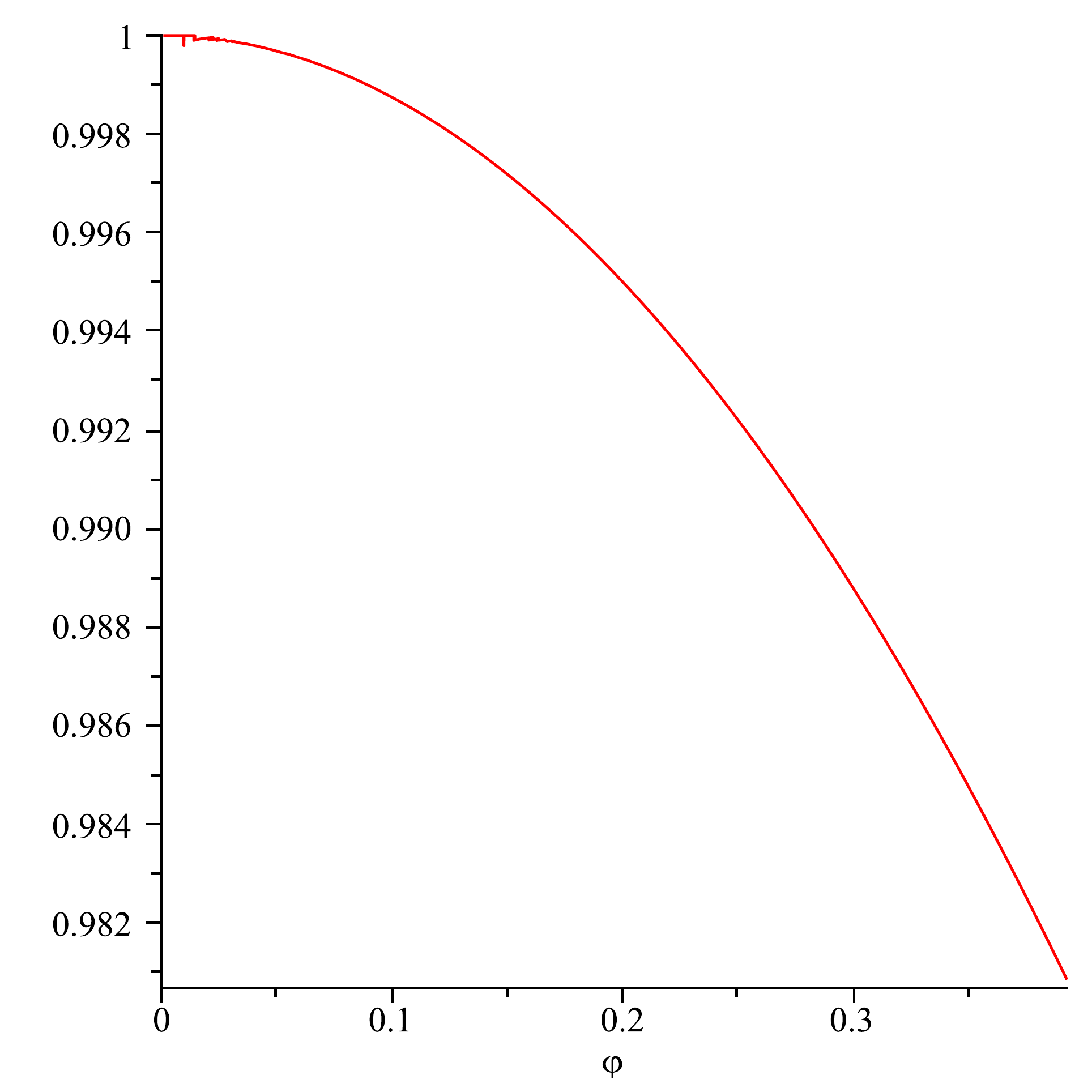}
 \caption{A depiction of $\frac{1}{\eE_{100}(0)}\eE_{100}(\vfi/100)$, $0\leq \vfi\leq \frac{\pi}{8}$.  The rescaling $\vfi\mapsto\frac{\vfi}{100}$ pushes  the oscillations outside the screen.}
 \label{fig: spec2}
 \end{figure}
\qed
\end{remark}

\appendix

\section{Sharp elliptic estimates}
\label{s: ell-est}
\setcounter{equation}{0}

The  main goal of this appendix is to provide a proof of the  estimates (\ref{eq: control}). We will  follow    the same  strategy  employed  in the proof of \cite[Thm. 17.5.3]{Ho3}. We  will reach  a  conclusion  stronger than  \cite[Thm. 17.5.3]{Ho3} because we will rely on   \emph{sharp} a priori elliptic estimates.   We obtain  these sharp estimates from  the precise a priori estimates for elliptic   operators with real analytic  coefficients  in \cite[Chap. 8]{LM}.    For the reader's convenience we first  synchronize our notations with those in  \cite{LM}.

For any  nonnegative integers $s$ we set $M_s:= s!$. Observe that  \cite[Chap. 8, Eq. (1.10)]{LM} becomes
\begin{equation}
M_{t+s}\leq d^{t+s} M_tM_s,\;\;\forall s,t\in\bZ_{\geq 0},\;\;d:=2.
\label{eq: e1}
\end{equation}
 
Denote by  $\Delta$    the Laplacian operator defined by a real analytic metric $g$ on the   ball  $B_4$, where 
 \[
 B_{R}=\bigl\{ x\in \bR^m;\;\; |x|<R\,\bigr\}.
 \]
   For  any $R>0$ we denote by $\eA(B_R)$ the space of  functions real analytic on the ball  $B_R$. We set
   \begin{equation}
   \|u\|_{k,R}=\sum_{|\alpha|=k}\left(\int_{B_R} |\pa^\alpha_x u(x)|^2 |dx|\,\right)^{1,2}.
   \label{eq: sobo}
   \end{equation}
   
  \begin{theorem}  
  Then there exist constants $K,T>0$, $\rho_0\in (0,1) $ that  depend only on the real analytic metric $g$  and satisfying the following properties. If  $u\in \eA(B_2)$  and     $Z_0,L_0>1$ are  such that
  \begin{equation}
  \|\Delta^k u\|_{0,2}\leq Z_0L_0^k M_{2k},\;\;\forall k\geq 0,
  \label{eq: e0}
  \end{equation}  
  then 
  \begin{equation}
  \|u\|_{j,\rho_0}\leq  Z_0 \bigl(\, T\sqrt{L_0}\,\bigr)^j M_j,\;\;\forall j\geq 0.
  \label{eq: fundamental}
  \end{equation}
  \begin{equation}
  |\pa^\alpha_xu(0)|\leq K Z_0L_0^{\frac{m}{4}}\bigl(\, 2T\sqrt{L_0}\,\bigr)^j M_j ,\;\;\forall j\geq0,\;\forall |\alpha|=j. 
  \label{eq: fundamental2}
  \end{equation}
  We want to emphasize that $K$ and $T$ are \emph{independent} of $u$, $L_0$, $Z_0$.
  \label{th: fund}
  \end{theorem}
  
 \begin{proof} We follow closely the approach in \cite[Chap. 8, Sec. 2]{LM}.    From   Theorem 2.1 op. cit. we deduce  the following.
 
 \begin{lemma} There exist $\rho_1,C_1>0$, $\rho_1<1$,  such that if $0<\rho<\rho+\delta<\rho_1$  and any $u\in \eA(B_2)$ we have 
 \begin{equation}
 \|u\|_{2,\rho} \leq C_1\left(\|\vfi_{\rho,\delta}\Delta u\|_{0,\rho+\delta}+\sum_{\ell=0}^{2m-1}\frac{1}{\delta^{2m-\ell}}\|u\|_{\ell,\rho+\delta}\right).
 \label{eq: e2}
 \end{equation}
 \qed
 \label{lemma: e1}
 \end{lemma}
 
 Arguing as in   the proof of  \cite[Chap. 8, Lemma 2.5]{LM} we obtain the following result. 
 \begin{lemma} Suppose  that $a\in \eA(B_2)$   satisfies 
 \[
 \sup_{X\in B_3} \sum_{|\alpha|=r} |\pa^\alpha a(x)|\leq L^r M_r,\;\;\forall r\geq 0.
 \]
 Then for every $r,s\in \bZ_{\geq 0}$, $0<\rho<\rho+\delta<1$ and $u\in \eA(B_2)$ we have 
 \begin{equation}
 \begin{split}
 \sum_{|\beta|=s}\sum_{|\alpha|=r} \| \pa^\beta[\vfi_{\rho,\delta}(a \pa^\alpha u-\pa^\alpha(au)\|_{0,\rho+\delta}\\
 \leq  C^*_s\sum_{\ell=0}^s\frac{1}{\delta^\ell} \sum_{t=0}^{s-\ell} L^{r+t}M_{r+t} \sum_{i=0}^{r-1}\frac{1}{L^iM_i} \|u\|_{s-\ell-t+i,\rho+\delta},
 \end{split}
 \label{eq: e3}
 \end{equation}
 where $C^*_s$ depends only on $s$.\qed
 \label{lemma: e2}
 \end{lemma}
 
 Using  Lemma \ref{lemma: e1} and \ref{lemma: e2} we deduce as in \cite[Chap. 8, Lemma 2.6]{LM} the following result.
 
 \begin{lemma}  Suppose that $\rho_1, C_1$ are as in Lemma \ref{lemma: e1}. For any $\ve>0$  there exists $\gamma(\ve)>0$  that depends only on the metric and $\ve$ such that  for any  $u\in \eA(B_2)$, $k>0$ and $\rho<\rho+\delta <\rho_1$ we have 
 \begin{equation}
 \begin{split}
 \|u\|_{2k+2,\rho}\leq C_1^* \Biggl\{ \|\Delta u\|_{2k,\rho+\delta}+\ve\|u\|_{2k+2,\rho+\delta}  \\
  M_{2k} \sum_{s=0}^{k=1}\frac{\gamma(\ve)^{k-s}}{M_{2s}} \|u\|_{2s+2,\rho+\delta}+\frac{M_{2k}}{\delta^2} \sum_{s=-1}^{k-1}\frac{\gamma(\ve)^{k-s-1}}{M_{2s+2}} \|u\|_{2s+2,\rho+\delta}\,\Biggr\},
 \end{split}
 \label{eq: e4}
 \end{equation}
 where the constant $C_1$ depends only on the metric $g$.\qed
 \label{lemma: e3}
 \end{lemma}
 
 For every $\lambda, R>0$ such that $R<\rho_1$ we set
 \[
 \si^k(u,\lambda, R):=\frac{1}{M_{2k}\lambda^{k+1}} \sup_{R/2\leq \rho <R} (R-\rho)^{2k+2}\|u\|_{2k+2,\rho}.
 \]
 Using Lemma \ref{lemma: e3}  as in the proof of \cite[Chap. 8, Lemma 2.7]{LM} we obtain  the following result.

 \begin{lemma} There exists $\lambda_1>0$ such that for any $R<\frac{1}{2}\rho_1$, $\lambda\geq \lambda_1$, $k\geq 0$  and every $u \in \eA(B_2)$ we have 
 \begin{equation}
 \si^k(u,\lambda, R)\leq \frac{M_{2k-2}}{4M_{2k}}\si^{k-1}(\Delta u,\lambda R) +\frac{1}{4}\sum_{s=-1}^{k-1} \si^s(u,\lambda, R).
 \label{eq: e5}
 \end{equation}
 \qed
 \label{lemma: e4}
 \end{lemma}
 An argument identical  to the one used  in the  proof of \cite[Chap. 8, Thm. 2.2]{LM}  and based on Lemma \ref{lemma: e4}   shows  that if  $u\in \eA(B_2) $    satisfies (\ref{eq: e0}), then for $\lambda, R$ as in Lemma \ref{lemma: e4} we have
  \begin{equation}
 \si^k(u,\lambda, R)\leq \frac{M_{2k+2}}{M_{2k}}Z_0(4L_0+2)^{k+1}.
 \label{eq: e6}
 \end{equation}
 If we let $\lambda=\lambda_1$  and $R_0=\frac{1}{4}\rho_1$  in the above   inequality we deduce
 \[
 \frac{1}{\lambda_1^{k+1}}\sup_{R_0/2 \leq \rho< R_0}(R_0-\rho)^{2k+2}\|u\|_{2k+2, \rho}\leq  M_{2k+2}Z_0(4L_0+2)^{k+1}
 \]
 In particular, if $\rho_0=3R_0/4$  and $\mu:=\sqrt{\lambda_1}$ we deduce
 \[
 \|u\|_{2k+2, \rho_0} \leq c_0 M_{2k+2}\left(\frac{4\mu}{R_0}\right)^{2k+2}(4L_0+2)^{k+1}= Z_0 M_{2k+2} \zeta^{2k+2},\;\;\zeta:=\frac{4\mu}{R_0}\sqrt{4L_0+2}.
 \]
We can assume that $\zeta>1$.

 Using the interpolation inequality  \cite[Chap. 8,  Eq. (2.7)]{LM} we deduce that
 \[
 \|u\|_{2k+1,\rho_0} \leq \zeta^{-1}\|u\|_{2k+2,\rho_0} +c_m\zeta \|u\|_{2k,\rho_0},
 \]
 where $c_m>0$ is a constant that depends only on the dimension $m$. Hence
 \[
 \|u\|_{2k+1,\rho_0}\leq  Z_0 M_{2k+2} \zeta^{2k+1} +c_mZ_0 M_{2k}\zeta^{2k+1}
 \]
 \[
 =Z_0M_{2k+1}\zeta^{2k+1}\Bigl( (2k+2) +\frac{c_m}{(2k+1)}\Bigr)\leq Z_0M_{2k+1}\bigl(\,(2+c_m)\zeta\,\bigr)^{2k+1}.
 \]
 We deduce that for any $k\geq 0$ we have
 \begin{equation}
 \|u\|_{k,\rho_0} \leq Z_0 M_k  Z^k,\;\; Z= (2+c_m)\zeta.
 \label{eq: e7}
 \end{equation}
 This proves (\ref{eq: fundamental}) with 
 \[
T=\frac{Z}{\sqrt{L_0+1}}.
 \]
 Using the Sobolev lemma \cite[Thm. 3.5.1]{Mo}  we deduce\footnote{We have to warn the reader that    the Sobolev  norms  $\|\nabla^jv\|_{2, R}$ defined  in \cite{Mo} do not coincide with the ones defined in (\ref{eq: sobo}) but they are equivalent. The constants implied by this equivalence of norms depend only on $j$ and $m$.} that there exists a constant  $K_m>0$ that depends only on $m$  such that for any $v\in C^\infty(B_{r/2})$ we have 
 \[
 |v(0)| \leq\frac{K_m}{r^{\frac{m}{2}}}\left( \sum_{j=0}^{p-1}\frac{r^j}{j!}\|v\|_{j,r}+ \frac{r^p}{(p-1)!} \|v\|_{p,\rho_0}\right),\;\; p=p_m:=\left\lfloor \frac{m}{2}\right\rfloor +1.
 \]
 We deduce that for any multi-index $\alpha$ such that $|\alpha|=k$  and any $r$ such that $r/2<\rho_0$ we have
 \[
 |\pa^\alpha u(0)|\leq \frac{K_m}{r^{\frac{m}{2}}}\left( \sum_{j=0}^{p-1}\frac{r^j}{j!}\|v\|_{j+k,\rho_0}+ \frac{r^p}{(p-1)!} \|u\|_{p+k,\rho_0}\right)
 \]
 \[
 \stackrel{(\ref{eq: e7})}{\leq} \frac{K_m Z_0}{r^{\frac{m}{2}}} \left( \sum_{j=0}^{p-1}\frac{r^j(k+j)!}{j!}Z^{k+j}+ \frac{r^p (k+p)!}{(p-1)!} Z^{k+p}\right)
 \]
 \[
 = \frac{K_m Z_0 k! Z^k}{r^{\frac{m}{2}}}\left( \sum_{j=0}^{p-1}\binom{k+j}{j}(rZ)^{j}+ p\binom{k+p}{p} (rZ)^{p}\right).
 \]
 Recall that $Z=T\sqrt{L_0+1}$. Choose $r$ of the form
 \[
 r=\frac{1}{B\sqrt{L_0+1}} < \frac{1}{B\sqrt{2}}.
 \]
  where $B$ is sufficiently  large so that
  \[
  \frac{r}{2}=\frac{1}{2B\sqrt{L_0+1}} < \frac{1}{2B\sqrt{2}}<\rho_0,
  \]
  and 
 \[
 (rZ)=\frac{T}{B}\leq \frac{1}{2}.
 \]
  \[
 |\pa^\alpha u(0)|\leq  K_m B^m Z_0(L_0+1)^{\frac{m}{4}} Z^k k! \left( \sum_{j=0}^{p-1}\binom{k+j}{j}2^{-j}+ p\binom{k+p}{p} 2^{-p}\right)
 \]
 \[
 \leq p_mK_m B^m Z_0(L_0+1)^{\frac{m}{4}} Z^k k! \underbrace{\sum_{j=0}^\infty\binom{k+j}{j}2^{-j}}_{=2^{k+1}}
 \]
 \[
 =2 p_mK_m B^m Z_0(L_0+1)^{\frac{m}{4}}(2T\sqrt{L_0+1})^k k!
 \]
 This proves  (\ref{eq: fundamental2}) with $K= 2 p_mK_m B^m$.
 \end{proof}

 \begin{remark}    The above arguments extend  with no changes to the case when  the   metric  belongs to a Gevrey space, or more general to one of the spaces $\eD_{M_k}(B_4)$ of \cite{LM}, where the weights   $M_k$  are subject to the constraints  (1.6)-(1.11) in \cite[Chap.8]{LM}.\qed
 \end{remark}

 Theorem \ref{th: fund} has the following immediate  consequence.

 \begin{corollary} Suppose $N$ is a compact   real analytic manifold  of dimension $n$ and $g$ is a real analytic metric  on $N$ with injectivity radius $r(N)$. Denote     by  $\Delta_g$ the Laplace operator of the metric $g$ and by $\eA(N)$ the space of real analytic functions on $N$.  Then there exist constants $K,T>0$ and $0<\rho_0<r(N)$ depending only on  $g$  with the following property: for any $Z_0, L_0>1$ and any  $u\in \eA(N)$  such that, if 
 \begin{equation}
 \|\Delta_g^k u\|_{L^2(N,g)}\leq Z_0 L_0^k k!,\;\;\forall k\geq 0,
 \label{eq: e10}
 \end{equation}
 then
 \begin{equation}
  |\pa^\alpha_xu(\bp)|\leq  K Z_0 L_0^{\frac{m}{4}}  \bigl(\, 2T\sqrt{L_0}\,\bigr)^jj!,\;\;\forall |\alpha|=j,\;\;\forall \bp\in M,
 \label{eq: e11}
 \end{equation}
 where  $x=(x^1,\dotsc, x^m)$ are normal coordinates  at $\bp$.\qed
 \label{cor: fund}
 \end{corollary}

 \noindent {\bf  Proof of  (\ref{eq: control})}.     We follow the same strategy   used in the proof of \cite[Thm. 17.5.3]{Ho3}. Fix $L>1$ and denote by $P_L$ the orthogonal projection  onto the space
 \[
 H_L:=\bigoplus_{\lambda\leq L} \ker (\lambda-\Delta).
 \]
 Fix $j,\ell \geq 0$ points $\bp_0,\bq_0$ in $M$, and normal coordinates $x$ at $\bp_0$ and $y$ at $\bq_0$. Let $f\in  L^2(M)$ and set $f_L=P_L f$.  Then
 \[
 \|\Delta^k f_L\|_{L^2(M)}\leq L^k \|f\|,\;\;\forall  k\geq 0.
 \]
 Using  Corollary \ref{cor:  fund}  we deduce   that 
 \[
 |\pa^\alpha_x f_L (\bp_0)| \leq K L^{\frac{m}{4}+\frac{j}{2}} (2T)^j j! \|f\|,\;\;\forall \bp\in M,\;\;|\alpha|=j,
 \]
 where $K, T$ are independent of $f$, $k$, $\alpha$ and $L$. Now observe that
 \[
 \pa^\alpha_x f_L (\bp_0)= (f, g_0)_{L^2(M)},\;\;  g_0(\bq)= \pa^\alpha_x \eE_L(\bp_0,\bq)=\sum_{\lambda_n\leq L} \bigl(\, \pa^\alpha_x \Psi_n(\bp_0)\,\bigr)\Psi_n(\bq).
 \]
The above discussion shows that for every $f\in L^2(M)$ we have
\[
 \bigl|\, (f,g_0)_{L^2(M)}\,\bigr|\leq K L^{\frac{m}{4}+\frac{j}{2}} (2T)^j j! \|f\|,
 \]
so that
\[
\|g_0\|_{L^2(M)}\leq K L^{\frac{m}{4}+\frac{j}{2}} (2T)^j j!.
\]
 Observe that $g_0\in H_L$  and thus for any $k\geq 0$ we have
 \[
 \|\Delta^k g_0\|_{L^2(M)}\leq   L^k \|g_0\|_{L^2(M)}
 \]
 and Corollary \ref{cor: fund} implies that for any $|\beta|=\ell$ we have
 \[
|\pa^{\alpha}_x\pa^\beta_y\eE_L(\bp_0,\bq_0)|= |\pa^\beta_y g_0(\bq_0)|\leq K L^{\frac{m}{4}+\frac{\ell}{2}} (4T)^\ell \ell!\|g\|_0\leq K^2L^{\frac{m}{2}+\frac{j}{2}+\frac{\ell}{2}} T^{j+\ell}j!\ell !
\]
 The inequality   (\ref{eq: control}) follows by observing that
 \[
j!\ell !\leq \binom{j+\ell}{j} (j+\ell)!\leq 2^{j+\ell}(j+\ell)!.
\]
  \qed


\begin{thebibliography}{XXXXX}

\bibitem{AAR} G.E. Andrews, R. Askey, R. Roy: {\sl Special Functions}, Encyclopedia of Math. and its Appl., vol 71, Cambridge University Press, 2006.



\bibitem{Bin} X. Bin: {\sl  Derivatives of the spectral function and Sobolev norms of eigenfunctions on a closed Riemannian manifold}, Ann. Global. Analysis an Geometry, {\bf 26}(2004), 231-252.

\bibitem{DF} H. Donnely, C. Fefferman: {\sl Nodal sets of eigenfunctions on Riemannian manifolds}, Invent. Math. {\bf 93}(1988), 161-183.



\bibitem{Ho3} L. H\"{o}rmander: {\sl The Analysis of Partial Differential Operators III}, Springer Verlag, 1994.


\bibitem{KP} S.G. Krantz, H. R. Parks: {\sl A Primer of Real Analytic Functions}, 2nd Edition  Birkh\"{a}user, 2002.

\bibitem{Lai} P.T. Lai: {\sl Meilleures estimations asymptotiques des restes de la fonction spectrale et des valeurs propres  relatif au Laplacien}, Math. Scand. {\bf 48}(1981), 5-38.

\bibitem{LPS} H. Lapointe, I. Polterovich, Yu. Safarov: {\sl Average growth of the spectral function on a Riemannian manifold},  Comm. Part. Diff. Eqs., {\bf 34}(2009), 581-615.

\bibitem{LM} J.L. Lions,  E. Magenes: {\sl  Non-Homogeneous Boundary Value Problems and Applications.  Volume III}, Springer Verlag 1973. 

\bibitem{Mo} C.B. Morrey: {\sl Multiple Integrals in the Calculus of Variations}, Springer Verlag, 1966.



\bibitem{Mu}  C. M\"{u}ller: {\sl  Analysis of Spherical Symmetries in Euclidean Spaces},   Appl. Math. Sci. vol. 129, Springer Verlag, 1998.

\bibitem{Pee} J. Peetre: {\sl Remark on eigenfunction expansions for elliptic operators with constant coefficients}, Math. Scand. {\bf 15}(1964), 83-92

\bibitem{Rau} H. Rau: {\sl \"{U}ber eine asymptotische Darstellung der Jacobischen Polynome durch Besselsche Funktionen},  Math. Z., {\bf 40}(1936), 683-692.

\bibitem{SV}  Yu. Safarov, D. Vassiliev: {\sl The Asymptotic Distribution of Eigenvalues of Partial Differential Operators}, Translations of Math. Monographs, vol. 155, Amer. Math. Soc., 1997.

\bibitem{Sze} G. Szeg\"{o}: {\sl  Orthogonal Polynomials}, Colloquium Publ., vol 23,  Amer. Math. Soc.,  2003.

\bibitem{Zel} S. Zelditch: {\sl Local and global analysis of eigenfunctions}, Handbook of geometric analysis. No. 1, 545Ð658, Adv. Lect. Math. (ALM), 7, Int. Press, Somerville, MA, 2008. \href{http://front.math.ucdavis.edu/0903.3420}{arXiv:0903.3420}


\end{thebibliography}
\end{document}